\documentclass{elsarticle}
\usepackage{amsmath,amssymb,amsthm,graphicx}
\usepackage{enumitem}
\usepackage{hyperref}
\usepackage{fullpage}
\usepackage{float}
\usepackage{setspace}
\usepackage{graphicx, color, caption}
\usepackage{subfig}

\newtheorem{theorem}{Theorem}
\newtheorem{lemma}[theorem]{Lemma}
\newtheorem{corollary}[theorem]{Corollary}
\newtheorem{conjecture}[theorem]{Conjecture}

\begin{document}

\begin{frontmatter}

\title{The Last Temptation of William T. Tutte}

\author[SFU]{Bojan Mohar\fnref{mohar,mohar2}}
\ead{mohar@sfu.ca}

\author[SFU]{Nathan Singer}
\ead{nsinger@sfu.ca}

\address[SFU]{Simon Fraser University\\
 Department of Mathematics\\
 8888 University Drive\\
 Burnaby, BC, Canada\\}

\fntext[mohar]{Supported in part by an NSERC Discovery Grant R611450 (Canada),
   by the Canada Research Chair program, and by the
   Research Grant J1-8130 of ARRS (Slovenia).}
\fntext[mohar2]{On leave from:
    IMFM \& FMF, Department of Mathematics, University of Ljubljana, Ljubljana,
    Slovenia.}

\begin{abstract}
In 1999, at one of his last public lectures, Tutte discussed a question he had considered since the times of the Four Color Conjecture.  He asked whether the 4-coloring complex of a planar triangulation could have two components in which all colorings had the same parity.  In this note we answer Tutte's question to the contrary of his speculations by showing that there are triangulations of the plane whose coloring complexes have arbitrarily many even and odd components.
\end{abstract}


\end{frontmatter}

%\date{\today}

%\maketitle

\section{Introduction}

In \cite{TutteNotes}, Tutte asked if there existed a triangulation of the plane whose 4-coloring complex had more than one component of the same parity.  He based this question on numerous examples he examined over the course of many years.  One obvious possibility is the icosahedron, whose 4-colorings form ten Kempe equivalence classes.  However, Tutte found out that the icosahedron has a connected colouring complex.  Nonetheless, he was resistant to making this question into a conjecture that no such examples exist, since, in his view, ``the data are too few to justify a Conjecture.''  However, he asked: ``If anyone knows of any case of two components of the same parity, I would be glad to hear of it.''

In this short note, we provide examples of 4-connected triangulations of the plane whose 4-coloring complexes have arbitrarily many components with odd colorings and arbitrarily many components with even colorings. Although this seems to resolve the Tutte problem, we end up with a closely related conjecture, which is based on an extensive computation, and which claims that for every planar triangulation whose 4-coloring complex is disconnected has a component of even parity and one of odd parity (see Conjecture \ref{conj:new}).

\section{Coloring Complexes and the Parity of 4-Colorings}

Given a 4-colorable graph $G$, an independent set of vertices $A \subseteq V(G)$ is said to be a \emph{color class} if it is one of the color classes for some 4-coloring of $G$.  The \emph{$4$-coloring complex} $B(G)$ is the graph which has all the color classes of $G$ as its vertices and two vertices $C,D \in V(B(G))$ joined by an edge if the color classes $C$ and $D$ appear together in a 4-coloring of $G$.  These graphs were introduced by Tutte \cite{Tutte1}, who discussed their basic structure.  See Figure 2 for an example of a coloring complex.

A coloring complex may be disconnected.  To see which 4-colorings have their color classes in the same component of $B(G)$, the following notion can be used.  We say that two 4-colorings $f$ and $f^{\prime}$ of $G$ are adjacent, $f \sim f^{\prime}$, if they have a common color class.  Then we take the transitive closure of the adjacency relation and we say that two colorings are \emph{connected} if they are in the same equivalence class of the transitive closure.  Since each coloring corresponds to a 4-clique in $B(G)$, colorings are connected if and only if their cliques are in the same component of the graph $B(G)$.

In the same paper \cite{Tutte1}, Tutte introduced the notion of even and odd colorings, whose explanation requires a bit more background.  Given a 4-colorable graph $G$, consider a 4-coloring $f$ of $G$, whose color classes are $A, B, C, D$.  An \emph{AB-Kempe chain} in $G$ is a connected component of $G[A \cup B]$.  We will denote by $p(A,B)$ the number of $AB$-Kempe chains in $G$.  Notice that one can derive $2^{p(A,B)-1}-1$ new colorings of $G$ from $f$ by taking an arbitrary subset of $AB$-Kempe chains and swapping the colors on each $AB$-Kempe chain in the subset.  This operation is called an \emph{$AB$-Kempe change}.

Given these definitions, Tutte defines
\begin{equation}\label{eq:1}
  J_{A}(f) = p(A,B) + p(A,C) + p(A,D) - p(B,C) - p(B,D) - p(C,D).
\end{equation}
Observe that, for a given 4-coloring $f$, the parity of $J_{A}(f)$ equals the parity of $J_{X}(f)$, for each $X \in \{ B,C,D \}$.  This motivates calling a 4-coloring $f$ an \emph{even 4-coloring} if $J_{A}(f)$ is even and calling $f$ an \emph{odd 4-coloring} if $J_{A}(f)$ is odd.  Moreover, following Tutte \cite{Tutte1}, we can easily show that every coloring in the same connected component of the 4-coloring complex $B(G)$ has the same parity if $G$ is a triangulation of the plane.  We include a short proof for the sake of completeness.

Given a particular 4-colouring $f$ of a planar triangulation $T$, we will denote by $A,B,C,D$ its color classes, by $p(A,B)$ the number of $AB$-Kempe chains, and by $e(A,B)$ the number of edges in $G[A\cup B]$. Finally, we will denote by $\deg(A)$ the sum of the degrees of all the vertices in $A$.  The same notation applies for other choices of color classes.

\begin{theorem}[Tutte \cite{Tutte1}] \label{thm: Tutte}
Let $T$ be a triangulation of the plane with $n$ vertices and $A$ a color class of its $4$-coloring $f$. Then
\begin{equation}
\label{eq:2}
   J_{A}(f) = 2 \vert A \vert - \deg(A) + n - 3.
\end{equation}
\end{theorem}

\begin{proof}
Let $t=2n-4$ denote the number of facial triangles of $T$ and observe that the edges of the dual graph $T^*$ of $T$ that correspond to $AB$-edges and $CD$-edges form a perfect matching in $T^*$, which shows that
\begin{equation}
\label{eq:3}
   e(A,B) + e(C,D) = t/2 = n-2.
\end{equation}

Let us now consider the subgraph $T[A\cup B]$ and its faces in the plane. Since it has $|A|+|B|$ vertices, $e(A,B)$ edges, and $p(A,B)$ components, it is easy to see that it has precisely $e(A,B) - (|A|+|B|-p(A,B)) + 1$ faces. Each of the faces contains precisely one $CD$-Kempe chain, so we conclude that $p(C,D) = e(A,B) - (|A|+|B|-p(A,B))+1$. By exchanging the roles of $AB$ vs. $CD$, we get $p(A,B) = e(C,D) - (|C|+|D|-p(C,D)) +1$. By combining these two equations, using that $|A|+|B|+|C|+|D|=n$, and using (\ref{eq:3}), we obtain that
\begin{equation}
\label{eq:4}
   p(A,B)-p(C,D) = |A| + |B| - e(A,B) - 1.
\end{equation}
Finally, by using the same equations for $p(A,C)-p(B,D)$ and $p(A,D)-p(B,C)$ and taking their sum, we obtain:
\begin{equation}
\label{eq:5}
   J_{A}(f) = 3|A| +|B|+|C|+|D| - (e(A,B)+e(A,C)+e(A,D)) - 3 =
   2 \vert A \vert - \deg(A) + n - 3.
\end{equation}
This completes the proof.
\end{proof}

The following corollary then immediately follows from the equation (\ref{eq:2}).

\begin{corollary} \label{cor: Parity}
If $f$ and $g$ are two 4-colourings of $T$ with a common colour class $A$, then $J_{A}(f) = J_{A}(g)$.
\end{corollary}

Now, by Corollary~\ref{cor: Parity}, any two 4-colorings in the same connected component of $B(T)$ will have the same parity.  Consequently, we will say that components of $B(T)$ are \emph{even} or \emph{odd} according to the parity of their colourings.

\section{4-Colorings of Triangulated Surfaces}

Fisk \cite{Fisk1} introduced an alternative view of the parity of 4-colorings.  We show how to reconcile this approach with Tutte's definition.

Given a triangulation $T$ of an orientable surface with one of its orientations fixed, we can view a 4-coloring $f$ of $T$ as a simplicial mapping onto the boundary of the tetrahedron, which will be denoted as $K_{4}$.  For each triangle $T_{i,j,l}$ of $K_{4}$, we consider the facial triangles in $T$ that are mapped onto $T$ in such a way that their orientation is preserved and those facial triangles whose orientation is reversed.  Let $t_{ijl}^{+}$ $(t_{ijl}^{-})$ be the number of triangular faces in $T$ that are mapped onto $T_{ijl}$ with positive (negative) orientation.  Then the value
\[ \deg(f) = t_{ijl}^{+} - t_{ijl}^{-} \]
is independent of the choice of the $ijl$ and is called the (\emph{homology}) \emph{degree of the mapping} $f : T \rightarrow K_{4}$.  Of course, modulo 2, $\deg(f)$ is equal to $t_{ijl} = t_{ijl}^{+} + t_{ijl}^{-}$, which is the number of triangular faces of $T$ with colors $i, j, l$.

If $T$ is a triangulation of a non-orientable surface, then we cannot define the values $t_{ijl}^{+}$ and $t_{ijl}^{-}$, but we can still define the parity (modulo 2) value of $t_{ijl}$ as $\deg(f)$.

Now, reinterpreting Tutte in the context of higher genus surfaces, we first define the \emph{Euler genus} of a surface triangulation $T$ to be equal to twice the genus of the surface if it is orientable, and be equal to the crosscap number of the surface if it is non-orientable. Then we define
\[ J_{A}(f) = 2 \vert A \vert - \deg(A) + n - 3 + g \]
where $n=|V(T)|$ and $g$ is the Euler genus of $T$.

\begin{theorem}
\label{thm: genusparity}
Let $T$ be a triangulation of a surface of Euler genus $g$ with $n=|V(T)|$ vertices. If $f$ is a $4$-coloring of $T$ and $A$ is one of its color classes, then
\begin{equation*}
  J_{A}(f) \equiv \deg(f) + n - 3 + g \pmod 2,
\end{equation*}
Consequently, the homology degree $\deg(f)$ has the same parity as $\deg(A)$.
\end{theorem}

\begin{proof}
We may assume that the color class $A$ corresponds to the color 4.  Let $T_{123}$ be the triangle of $K_{4}$ corresponding to the other 3 colors.  Then $\deg(f) \equiv t_{123} \pmod 2$. On the other hand,
\begin{equation*}
\deg(A) = \sum_{v \in A} \deg(v) = t_{124} + t_{134} + t_{234}
= \vert F(T) \vert - t_{123},
\end{equation*}
where $F(T)$ denotes the set of the triangular faces of $T$. Since $T$ is a triangulation, Euler's formula implies that $|F(T)| = \tfrac{2}{3}|E(T)| = 2n-4+2g$.
Thus,
\begin{flalign*}
J_{A}(f) & = 2 \vert A \vert - \deg(A) + n - 3 + g \\
& = 2 \vert A \vert + t_{123} - \vert F(T) \vert + n - 3 + g \\
& \equiv t_{123} + n - 3 + g \equiv \deg(f) + n - 3 + g \pmod 2.
\end{flalign*}
\end{proof}

\section{Coloring Complexes with Many Components}

We have found two triangulations of the plane on 12 vertices which answer Tutte's question in the affirmative.  Their 4-coloring complexes each have three components, from which it follows that two of these components must have the same parity.  These triangulations are drawn in Figure 1.

\begin{figure}[htb]
     \centering
     \subfloat[][Example 1]{\includegraphics[width=0.42\textwidth]{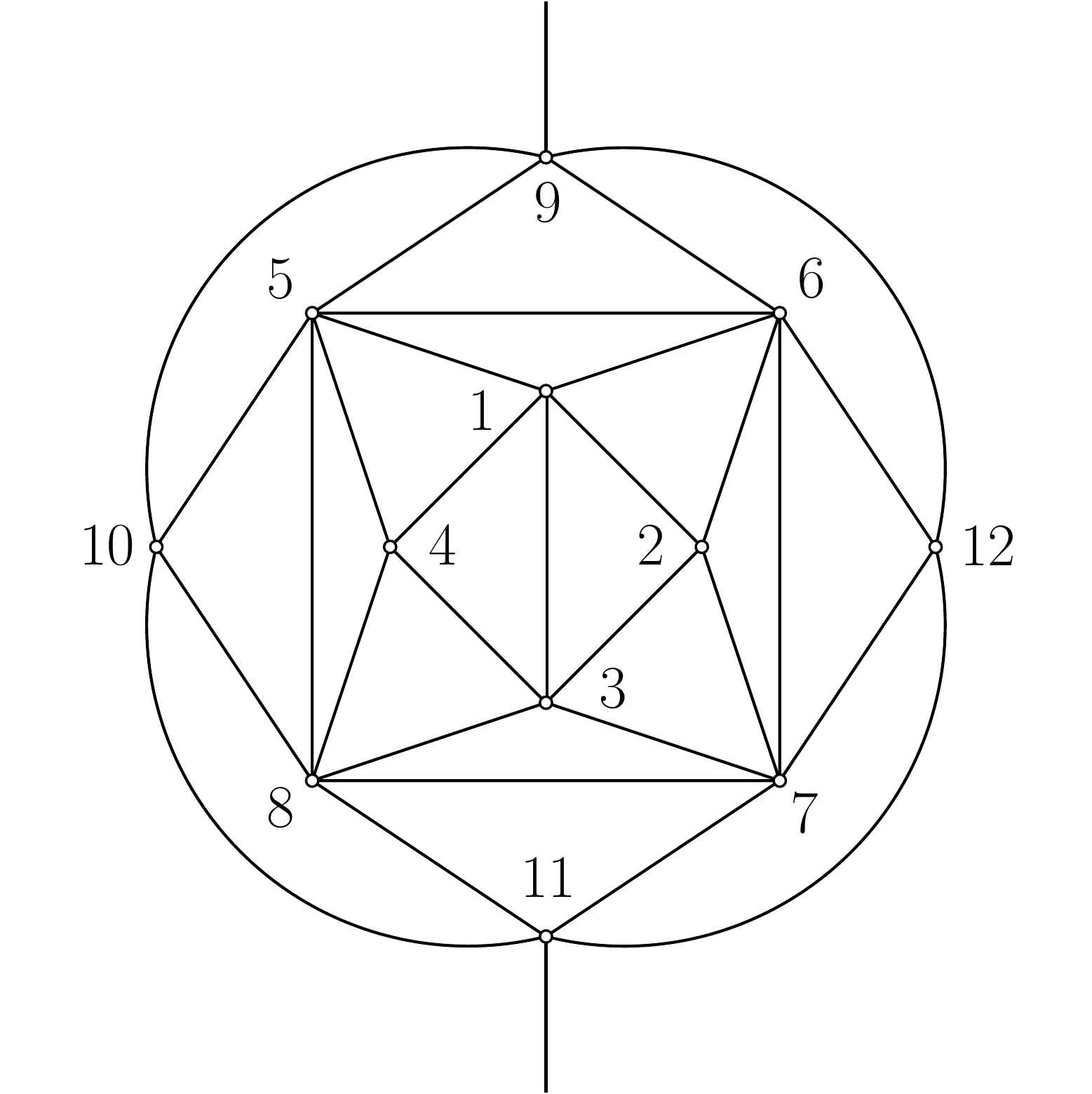}\label{<figure1>}}
     \qquad
     \subfloat[][Example 2]{\includegraphics[width=0.42\textwidth]{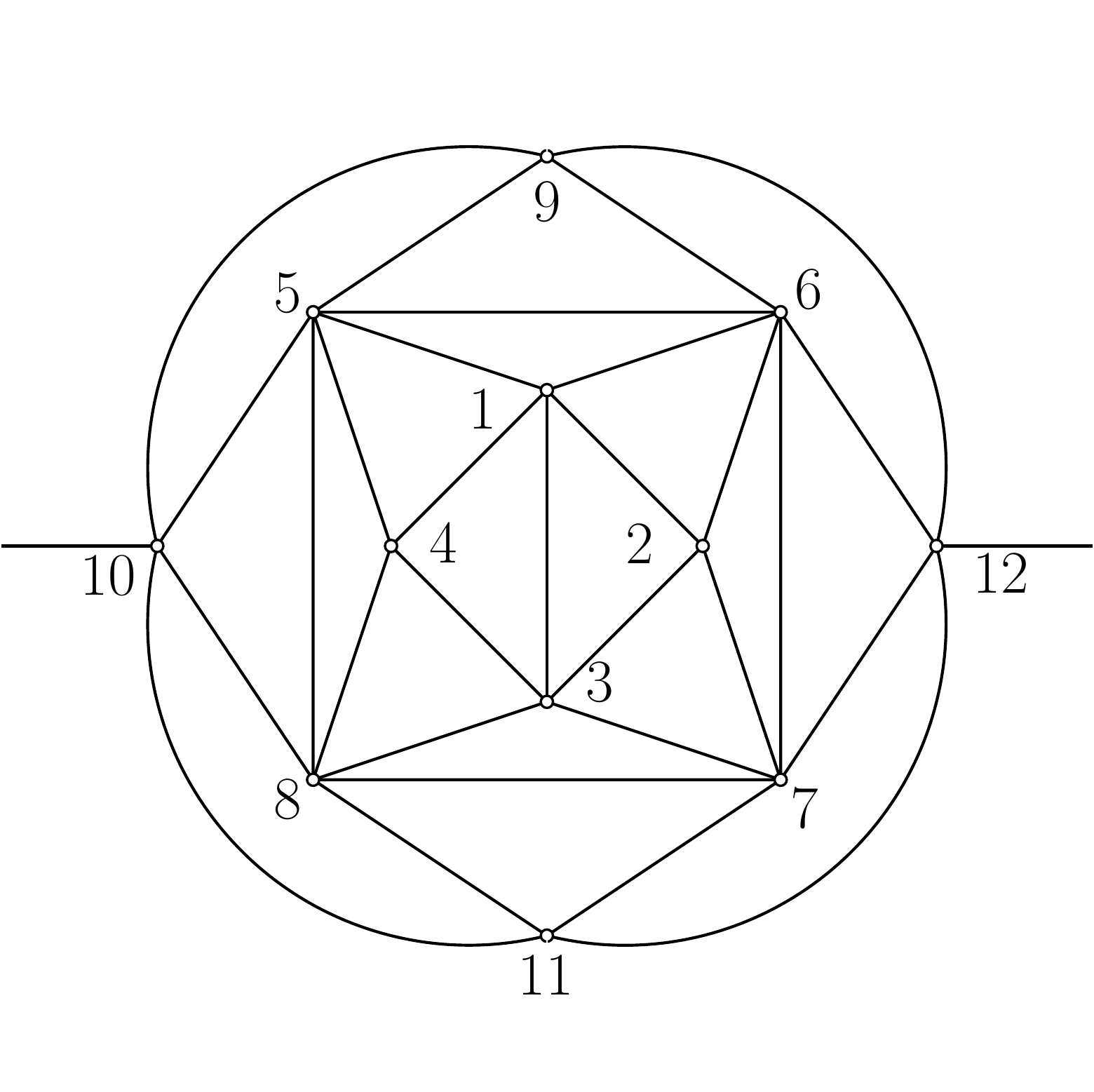}\label{<figure2>}}
     \caption{Triangulations whose 4-coloring complexes have three components.  The half-edges sticking out form the edges 9-11 and 10-12, respectively.}
\end{figure}

We will now determine the 4-coloring complex of Example 1, and show that it has 3 connected components (the argument for Example 2 is similar).

\begin{theorem} \label{thm: Example 1}
The 4-coloring complex of the first example $G$ given in Figure 1(a) is the graph drawn in Figure 2.
\end{theorem}

\begin{proof}
We need to determine all the 4-colorings $f$ of $G$.  We start by precoloring the 4-cycle with vertices labelled 5, 6, 7 and 8.  Up to relabelling, this 4-cycle can be precolored in 4 ways: \\[1mm]
(a) $f(5) = 1, f(6) = 2, f(7) = 3, f(8) = 4$. \\
(b) $f(5) = 1, f(6) = 2, f(7) = 1, f(8) = 2$. \\
(c) $f(5) = 1, f(6) = 2, f(7) = 1, f(8) = 3$. \\
(d) $f(5) = 1, f(6) = 2, f(7) = 3, f(8) = 2$. \\[1mm]
We consider these cases separately.

(a) Suppose that $f(1) = 3$.  Then $f(4) = 2$, $f(3) = 1$ and $f(2) = 4$.  Alternatively, if $f(1) = 4$, then $f(2) = 1$, $f(3) = 2$, and $f(4) = 3$.  Now, as $G$ is symmetric about the precolored 4-cycle, we also have exactly 2 ways of coloring the exterior of this 4-cycle.  As a consequence, we find that there are exactly 4 colorings satisfying (a). Their color classes are:\\
$\{ \{ 3, 5, 11 \}, \{ 4, 6, 10 \}, \{ 1, 7, 9 \}, \{ 2, 8, 12 \} \}$, $\{ \{ 3, 5, 12 \}, \{ 4, 6, 11 \}, \{ 1, 7, 10 \}, \{ 2, 8, 9 \} \}$, \\ $\{ \{ 2, 5, 11 \}, \{ 3, 6, 10 \}, \{ 4, 7, 9 \}, \{ 1, 8, 12 \} \}$, $\{ \{ 2, 5, 12 \}, \{ 3, 6, 11 \}, \{ 4, 7, 10 \}, \{ 1, 8, 9 \} \}$.

(b) We may assume that $f(1) = 3$.  Then $f(2) = 4$, and so vertex 3 has no available color.  Consequently, no 4-colorings of $G$ satisfy (b).

(c) Suppose that $f(1) = 3$.  Then $f(2) = 4$, so $f(3) = 2$.  Thus, $f(4) = 4$.  On the other hand, if $f(1) = 4$, then $f(2) = 3$.  Consequently, $f(3) = 2$, and vertex 4 cannot be colored.  Hence, this coloring does not occur.  Thus, we find that there is exactly one coloring satisfying (c): $\{ \{ 5, 7 \}, \{ 3, 6, 11 \}, \{ 1, 8, 9 \}, \{ 2, 4, 10, 12 \} \}$.

(d) This case is mirror-symmetric to (c).  We obtain exactly one coloring satisfying (d): \\
$\{ \{ 3, 5, 11 \}, \{ 6, 8 \}, \{ 1, 7, 9 \}, \{ 2, 4, 10, 12 \} \}$.

The six colorings of $G$ clearly make the coloring complex exhibited in Figure \ref{fig:2}.
\end{proof}

\begin{figure}[htb]
     \centering
     {\includegraphics[width=0.8\textwidth]{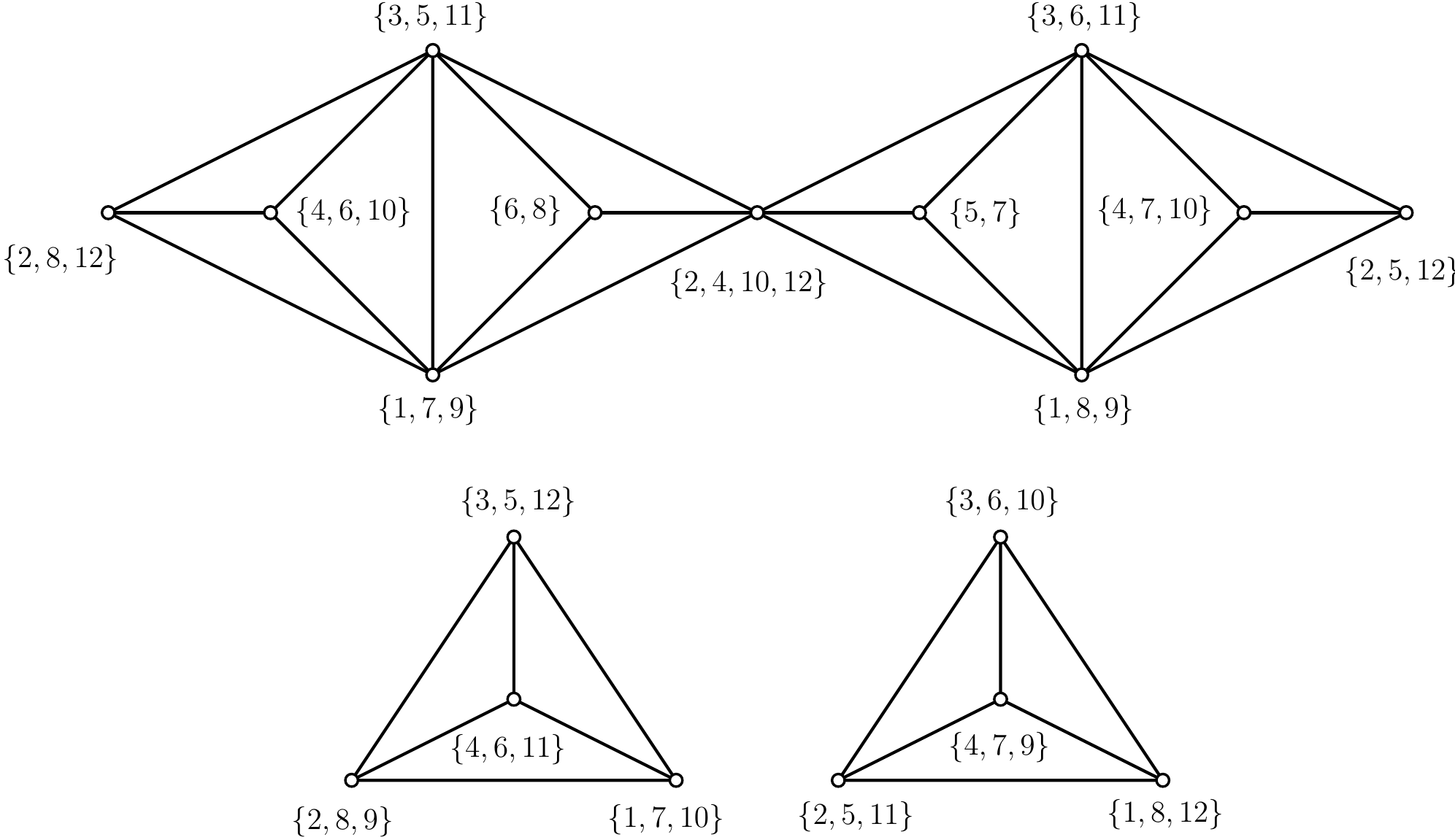}}
     \caption{The coloring complex $B(G)$}
     \label{fig:2}
\end{figure}

The coloring complex $B(G)$ from Theorem~\ref{thm: Example 1} certainly has three components, two of which have even parity, while the third one has odd parity.  This provides an affirmative answer to Tutte's question.  The triangulation shown in Figure 1(b) is not isomorphic to the previous example (this is easy to see by noting that the first graph has two degree 5 vertices with two common neighbors of degree 6, while the the second graph does not) but their coloring complexes are isomorphic.  It can be shown through computation that there are no such examples on 11 or fewer vertices. (Let us observe that using our Lemma \ref{lem: sum}, the computation can be reduced to 4-connected triangulations.)

We can now easily construct infinitely many examples by successively adding vertices to a triangle in $G$.  For example, consider the triangle 123 in $G$.  If we add vertex 13 in the centre of this triangle, and join it to vertices 1, 2 and 3 by edges, then we obtain a new triangulation $G^{\prime}$ of the plane such that $B(G) \cong B(G^{\prime}$).  As we can repeat this procedure indefinitely, this approach yields infinitely many examples of graphs which have two components with even parity.

Indeed, with a bit more effort, we can obtain a stronger result in this direction.

\begin{theorem} \label{thm: infinite}
Let $k \in \mathbb{N}$.  Then there are infinitely many triangulations of the plane, whose coloring complex has more than $k$ components with even colorings and more than $k$ components with odd colorings.
\end{theorem}

This theorem follows directly from Theorem~\ref{thm: Example 1} and the following lemma.  However, this lemma requires some further notation.

Let $G$ and $H$ be triangulations of surfaces of Euler genus $g_1$ and $g_2$, respectively. Let $T_1 = xyz$ ($T_2 = x'y'z'$) be a face in $G$ ($H$).  By identifying $x$ with $x'$, $y$ with $y'$, and $z$ with $z'$, we obtain a new triangulation $G \bigtriangleup H$ with $\vert V(G) \vert + \vert V(H) \vert - 3$ vertices and with Euler genus $g_1+g_2$.  Any 4-coloring of $G \bigtriangleup H$ induces a 4-coloring $f$ on $G$ and a 4-coloring $h$ on $H$, and then we use $f \bigtriangleup h$ to denote the coloring of $G \bigtriangleup H$.  This gives a bijection between colorings of $G \bigtriangleup H$ and pairs $(f,h)$ of colorings of $G$ and $H$.

\begin{lemma} \label{lem: sum}
(a) If two 4-colorings $f \bigtriangleup h$ and $f^{\prime} \bigtriangleup h^{\prime}$ of $G \bigtriangleup H$ are adjacent in $B(G \bigtriangleup H)$, then $f \sim f^{\prime}$ in $B(G)$, and $h \sim h^{\prime}$ in $B(H)$.  Conversely, if $f \sim f^{\prime}$ in $B(G)$, and $h \sim h^{\prime}$ in $B(H)$, then $f \bigtriangleup h \sim f^{\prime} \bigtriangleup h \sim f^{\prime} \bigtriangleup h^{\prime}$ in $B(G \bigtriangleup H)$.  \\
(b) If $A$ is a color class of $f \bigtriangleup h$, then $J_{A}(f \bigtriangleup h) = J_{A \cap V(G)}(f)  + J_{A \cap V(H)}(h)$.
\end{lemma}

\begin{proof}
(a) By definition, $f \bigtriangleup h \sim f^{\prime} \bigtriangleup h^{\prime}$ means that $f \bigtriangleup h$ and $f^{\prime} \bigtriangleup h^{\prime}$ have a color class in common, say $A$.  Then $A \cap V(G)$ is a color class that is common to $f$ and $f^{\prime}$, so $f \sim f^{\prime}$.  Similarly, $h \sim h^{\prime}$.  The reverse statement is equally clear.

(b)  Consider the color class $A$ of $f \bigtriangleup h$ which is disjoint from the new triangle we created by identifying vertices in $G$ and $H$.  Recall that, by Corollary~\ref{cor: Parity}, computing $J_{A}(f \bigtriangleup h)$ for this color class is equivalent to computing $J_{C}(f \bigtriangleup h)$ for any color class $C$ of $f \bigtriangleup h$.  We write $A_{G}$ for $A \cap V(G)$ and $A_{H}$ for $A \cap V(H)$. Then:
\[ J_{A_{G}}(f) = 2 \vert A_{G} \vert - \deg_G(A_{G}) + n_1 - 3 + g_1, \]
\[ J_{A_{H}}(h) = 2 \vert A_{H} \vert - \deg_H(A_{H}) + n_2 - 3 + g_2. \]

Now, as $A_{G}$ and $A_{H}$ are disjoint, $\deg(A) = \deg_G(A_{G})+\deg_H(A_{H})$. We can use the above expressions to compute $J_{A}(f \bigtriangleup h)$:
\begin{flalign*}
J_{A}(f \bigtriangleup h) & = 2 \vert A \vert - \deg(A) + (n_1+n_2-3) - 3 + (g_1+g_2) \\
& = 2 (\vert A_{G} \vert + \vert A_{H} \vert) - \deg_G(A_{G}) - \deg_H(A_{H}) + (n_1-3+g_1) + (n_2-3+g_2) \\
& = J_{A_{G}}(f) + J_{A_{H}}(f).
\end{flalign*}

If the color class $A$ contains one of the identified vertices, the calculation is the same except that $|A| = |A_{G}| + |A_{H}| - 1$ and
$\deg(A) = \deg_G(A_{G}) + \deg_H(A_{H}) - 2$, which yields the same final outcome.
\end{proof}

\section{4-connected Examples}

Identification over a triangle produces triangulations that are not 4-connected. However, it may be that in the back of Tutte's mind was an implicit condition that the triangulations should be 4-connected. The following generalized examples show that one can also find 4-connected examples whose coloring complexes have many components (of the same parity). The two graphs $Q_1$ and $Q_1'$ in Figure 1 are just smallest examples in two infinite families of 4-connected triangulations.  We define $Q_0,Q_1,Q_2,\dots$ as follows. To obtain $Q_k$ we take $k+2$ nested 4-cycles $D_i=a_ib_ic_id_i$ ($i=0,1,\dots,k+1$). We connect $D_i$ with $D_{i+1}$ ($i=0,1,\dots,k$) as shown in Figure \ref{fig:Qk}. To obtain $Q_k$, we add the edges $a_0c_0$ and $a_{k+1}c_{k+1}$.
We also define $Q_k'$, which is obtained in the same way except that we add the edge $b_{k+1}d_{k+1}$ instead of $a_{k+1}c_{k+1}$.

\begin{figure}[htb]
     \centering
     {\includegraphics[width=0.42\textwidth]{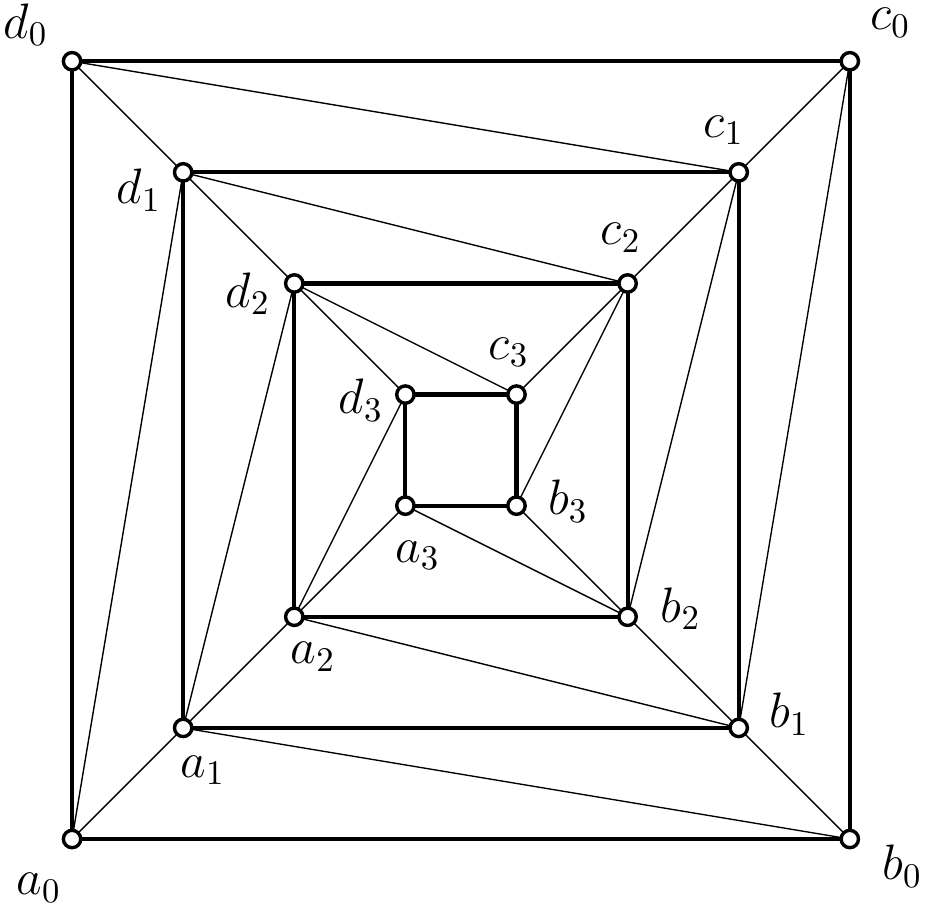}}
     \caption{Nested 4-cycles $D_0,D_1,D_2,D_3$. To obtain $Q_2$, we add the edges $a_0c_0$ and $a_3c_3$.}
     \label{fig:Qk}
\end{figure}

We say that a 4-coloring of $Q_k$ is of \emph{type I} if it uses all four colors on $D_0$, and of \emph{type II} if it uses only 3 colors on $Q_0$. Note that a coloring $f$ of type II has $f(b_0)=f(d_0)$ (and necessarily $f(a_0)\ne f(c_0)$).  As we saw in the proof of Theorem \ref{thm: Example 1}, every 4-coloring of type I on $D_i$ can be extended in two different ways to a 4-coloring on $D_{i+1}$, and both extensions are of type I (use all four colors on $D_{i+1}$ in the same clockwise order as on $D_i$). Similarly, every 3-coloring of type II on $D_i$ can be extended in two different ways to a 3-coloring on $D_{i+1}$, and both extensions are of type II. Of course, when $i=k$, only one of these 3-colorings of type II yields a 4-coloring in $Q_k$. This implies that $Q_k$ has precisely $2^{k+1}$ 4-colorings of type I and has precisely $2^k$ 4-colorings of type II.

Next, we can easily show that two 4-colorings of $Q_k$ of type I that have a color class in common must be identical (have all four color classes the same). In colorings of type I all color classes have a vertex in each $D_i$. A coloring of $Q_k$ of type II has two color classes that meet every $D_i$, while the other two color classes have two vertices in every second 4-cycle $Q_i$. Thus, if a color class of a type II coloring coincides with a color class of a type I coloring, then the color class is contained in precisely one coloring of type I and one coloring of type II. In this way, each coloring of type II has 2 color classes that coincide with two color classes of a 4-coloring of type I, and this correspondence is 1-1. The colorings $f$ of type I under this correspondence are precisely those fulfilling the following condition: either $\{f(b_1),f(d_1)\} = \{f(b_{k+1}),f(d_{k+1})\}$ (if $k$ is even) or $\{f(b_0),f(d_0)\} = \{f(b_{k+1}),f(d_{k+1})\}$ (if $k$ is odd). All of them are of odd parity and of even homology degree $\deg(f)$ (by Theorem \ref{thm: Tutte}, since the only vertices of odd degree are $a_0,c_0,a_{k+1}$ and $c_{k+1}$).

Finally, there are precisely $2^k$ 4-colorings of type I that do not fulfill the condition stated in the previous paragraph, and we can show that each of them forms a separate component in the coloring complex $B(Q_k)$. To see this, consider a color class $C$ of a coloring $f$ of type I. It has exactly one vertex in each of the nested 4-cycles $D_i$. By fixing the other three colors on vertices of $D_0$, it is easy to see that having a coloring on $D_i$ and knowing the vertex in $D_{i+1}\cap C$, the colors uniquely extend to $D_{i+1}$, for $i=0,1,\dots,k$. This shows that $f$ cannot share a color class with another coloring of type I. Thus, excluding the $2^{k+1}$ color classes in 4-colorings of type I which are Kempe equivalent to colorings of type II, we obtain $2^{k+2}$ color classes in the $2^{k}$ remaining 4-colorings of type I, each of which then forms a separate component in the coloring complex $B(Q_k)$.  Observe that all such colorings have even parity (odd homology degree).

It is also easy to see that all colorings of type II (together with their mates of type I) form a single component in $B(Q_k)$. To prove this, observe that two colorings of type II with a common class that intersects each $Q_i$ must be identical. However, if they have a common color class containing $b_0$ and $d_0$, then there are 2 ways to color each $D_i$ with $i$ odd, so there are $2^t$ colorings of type II with the same color class, where $t = \lfloor (k+1)/2 \rfloor$. And if they share a color class containing $a_1$ and $c_1$, or $b_1$ and $d_1$, then the same holds with $t = \lfloor k/2 \rfloor$.

In conclusion, $B(Q_k)$ has one (large) component containing all $2^{k+1}$ odd 4-colorings (they all have even homology degree) and $2^k$ components, each of which corresponds to a single coloring of even degree that is of type I.

The same holds for $Q_k'$.

\begin{theorem}
\label{thm:Qk}
The $4$-coloring complex of the $4$-connected triangulation $Q_k$ (and that of $Q_k'$) has $2^k + 1$ connected components. One of them corresponds to $2^{k+1}$ odd colorings, and each of the other $2^k$ components corresponds to a single even $4$-coloring (odd homology degree).
\end{theorem}

\section{Next Steps}

At present, we have not found 5-connected triangulations of the plane whose 4-coloring complexes would have arbitrarily many components.
However, we have found some 5-connected examples whose 4-coloring complexes have more than one component of the same parity.
The three smallest triangulations of this kind are illustrated in Figure \ref{fig: Three 5-Connected Triangulations on 20 Vertices}. These are triangulations of the plane on 20 vertices, and their minimality follows from an extensive computation that we performed on 5-connected triangulations with at most 23 vertices.

\begin{figure}[htp]
    \begin{minipage}{0.33\textwidth}
        \centering
        \includegraphics[width=0.8\textwidth]{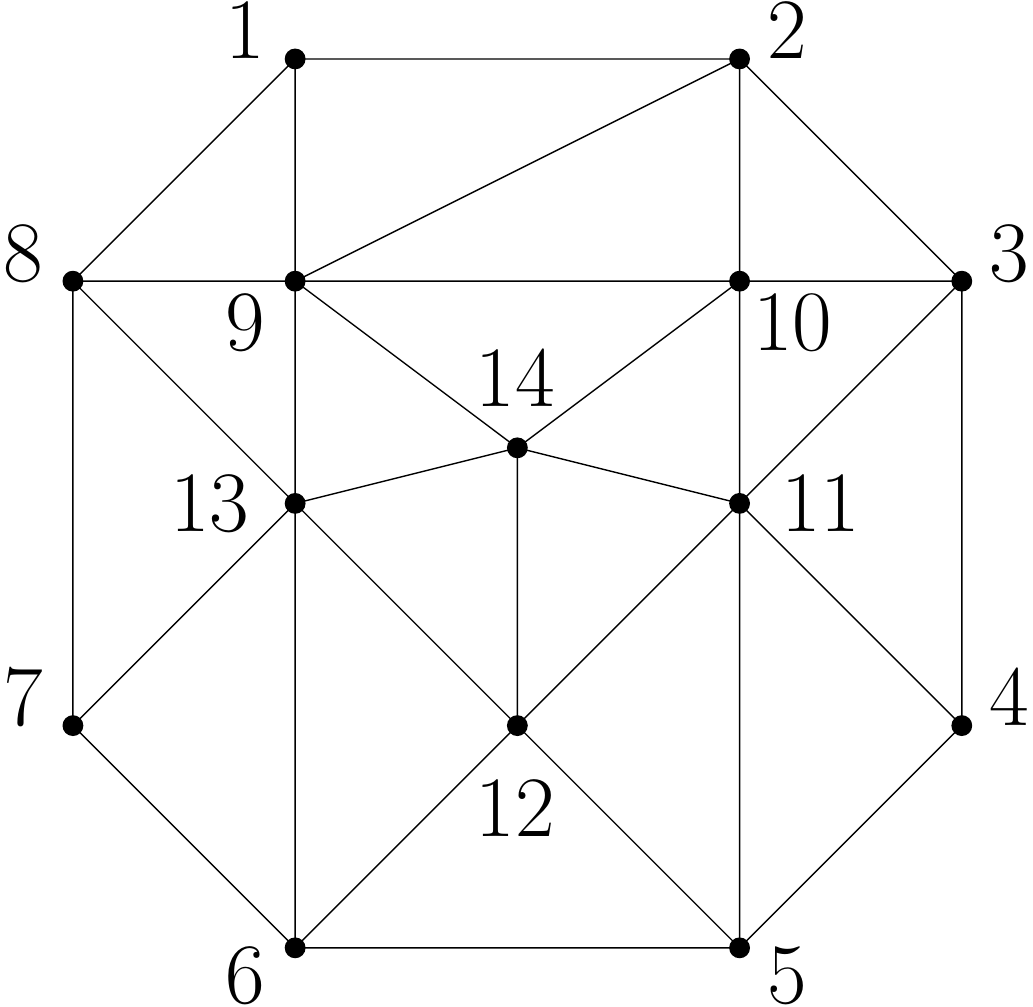}
    \end{minipage}%%
    {\hfill\color{black}\vrule\hfill}%
    \begin{minipage}{0.33\textwidth}
        \centering
        \includegraphics[width=0.8\textwidth]{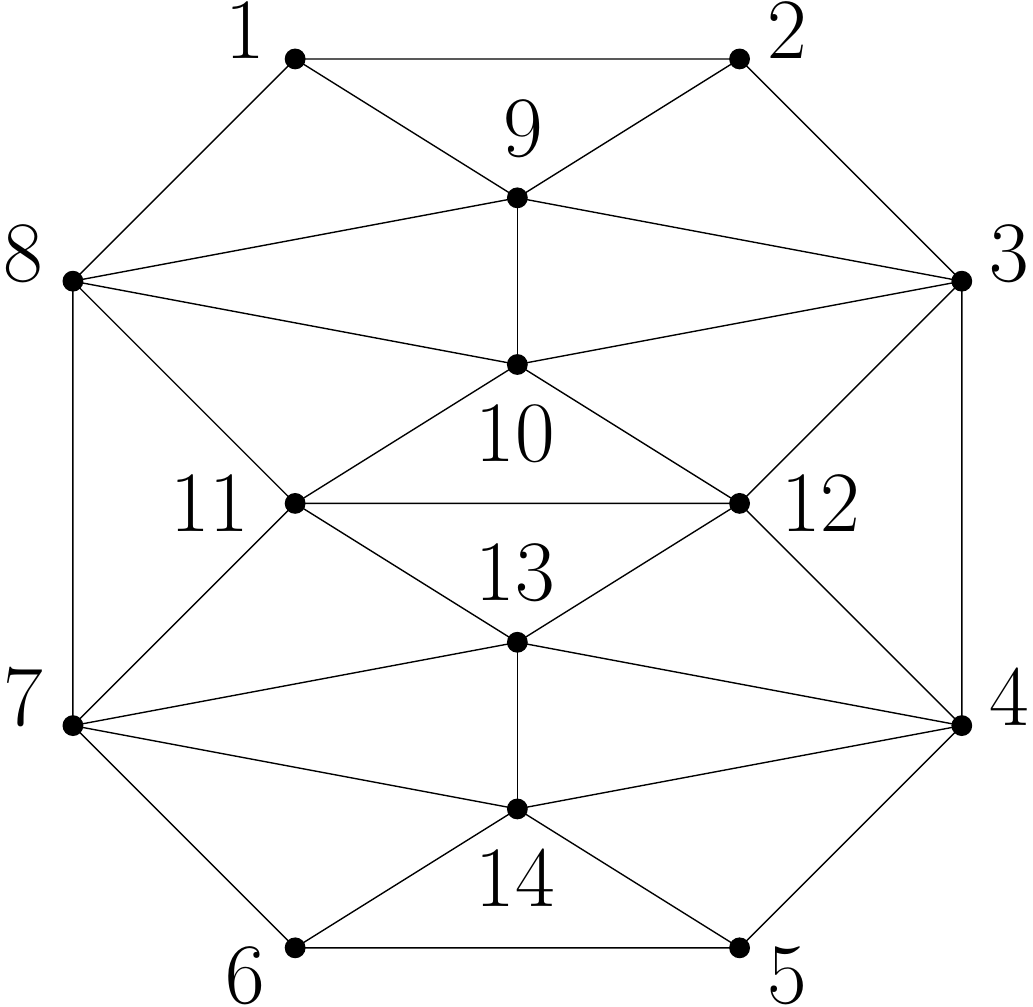}
    \end{minipage}%%
    {\hfill\color{black}\vrule\hfill}%
    \begin{minipage}{0.33\textwidth}
        \centering
        \includegraphics[width=0.8\textwidth]{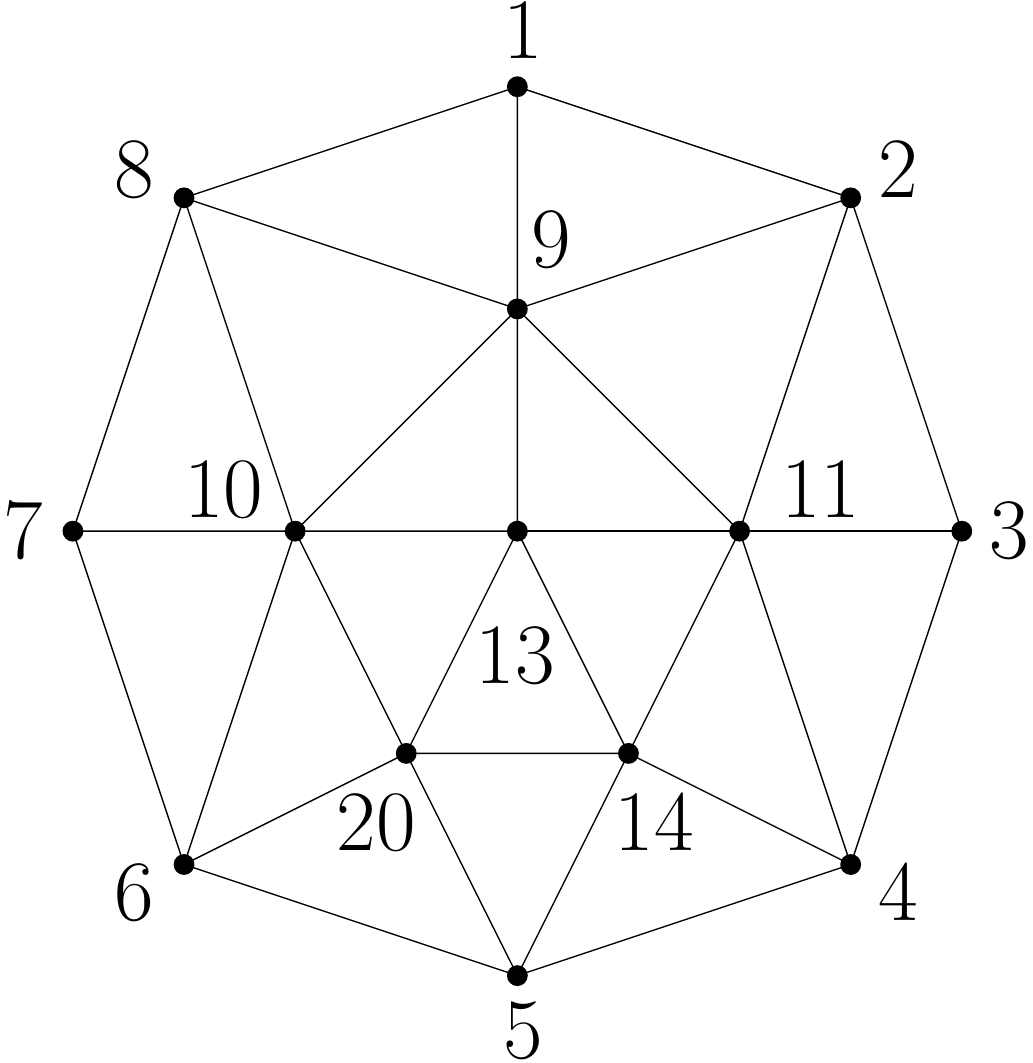}
    \end{minipage}%%

\medskip

    \begin{minipage}{0.33\textwidth}
        \centering
        \includegraphics[width=0.8\textwidth]{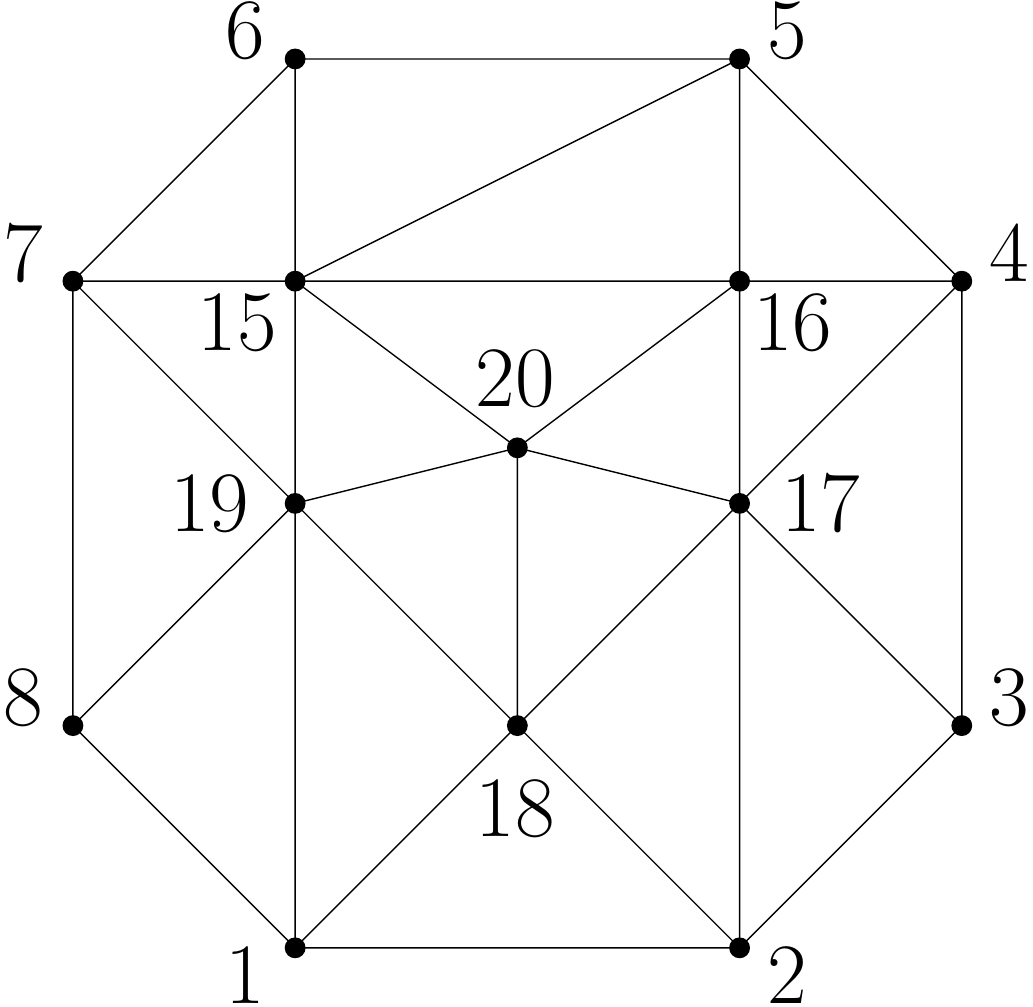}
    \end{minipage}%%
    {\hfill\color{black}\vrule\hfill}%
    \begin{minipage}{0.33\textwidth}
        \centering
        \includegraphics[width=0.8\textwidth]{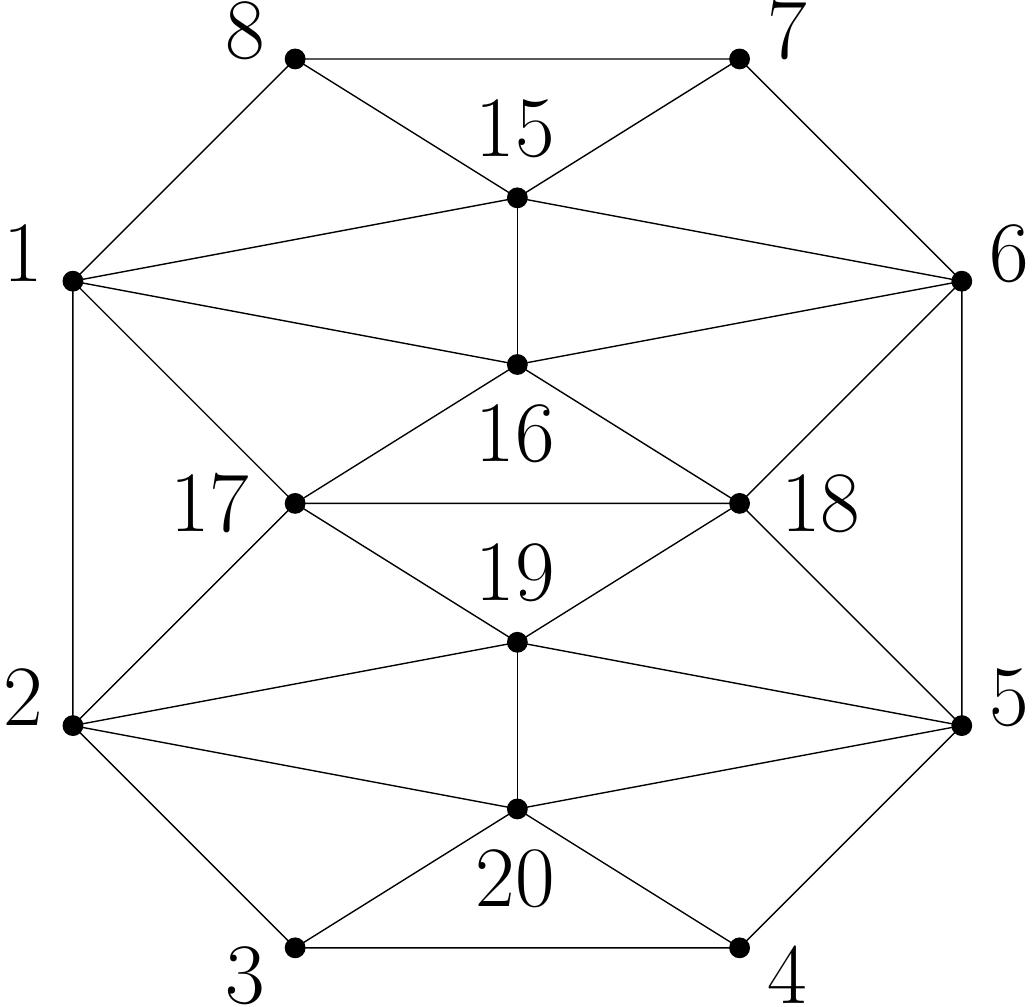}
    \end{minipage}%%
    {\hfill\color{black}\vrule\hfill}%
    \begin{minipage}{0.33\textwidth}
        \centering
        \includegraphics[width=0.8\textwidth]{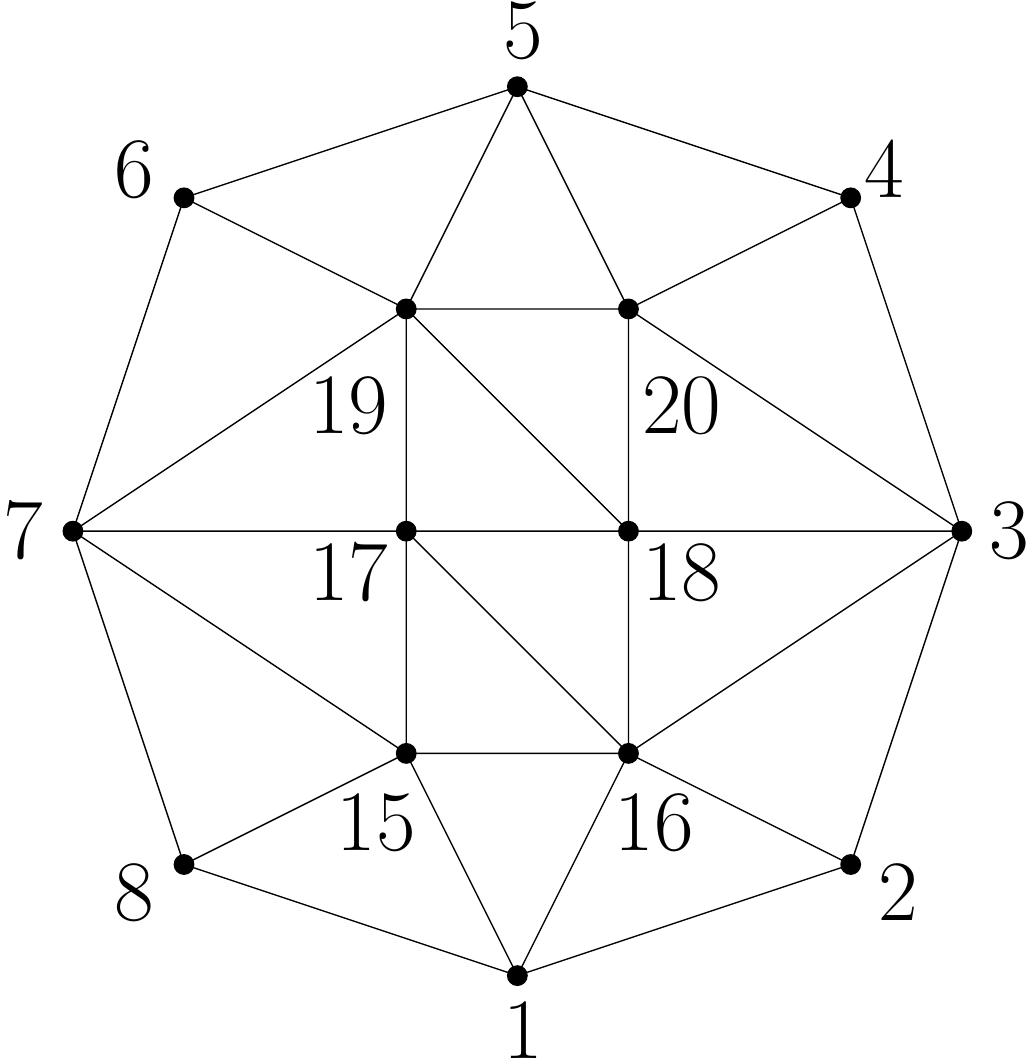}
    \end{minipage}%%
	\caption{Three 5-connected triangulations on 20 vertices. Each one is represented by two triangulated octagons. In order to obtain the full graph from each pair of subgraphs, identify the vertices on the outer cycle of each subgraph according to their labeling.}
	\label{fig: Three 5-Connected Triangulations on 20 Vertices}
\end{figure}

The 4-coloring complex of our first example has three components, two odd and one even.
The second example's 4-coloring complex has three even components and one odd component.
Our final example's 4-coloring complex has two even and one odd component.
There are also 6 non-isomorphic 5-connected examples on 21 vertices, 33 on 22 vertices and 66 on 23 vertices.
Unfortunately, we do not have an infinite family.

However, there is another interesting question about parity we can ask.
If the 4-coloring complex of a triangulation of the plane has at least two components, must it have one component of even parity and one component of odd parity?
Based on a large number of computations, we formulate the following conjecture.

\begin{conjecture}\label{conj:new}
Suppose that $T$ is a triangulation of the plane, and that its 4-coloring complex $B(T)$ has at least two components.  Then $B(T)$ has a component of even parity and a component of odd parity.
\end{conjecture}

\section*{References}

\bibliographystyle{plain}
\bibliography{references}

\end{document}